\documentclass[12pt]{article}
\usepackage{amsfonts,mathrsfs}
\usepackage{graphicx}
\usepackage{amsmath}
\usepackage{amsthm}
\usepackage{amssymb}
\usepackage{xparse}
\usepackage{color}
\usepackage[colorlinks]{hyperref}

    \newtheorem{theorem}{Theorem}
    \newtheorem{conjecture}[theorem]{Conjecture}
    
    \newtheorem{proposition}[theorem]{Proposition}

\theoremstyle{definition} 
    
    \newtheorem{remark}[theorem]{Remark}

\newcommand{\R}{{\mathbb R}}

\newcommand{\N}{{\mathbb N}}

\newcommand{\lstar}{{\raise-0.15ex\hbox{$\scriptstyle \ast$}}}

\theoremstyle{remark} 

\newcommand{\re}{\text{Re}}

\newcommand{\Exp}{\mathbb{E}}

\newcommand{\prob}{\mathbb{P}}

\renewcommand{\Pr}{\prob}
\DeclareDocumentCommand \one { o }
{%
\IfNoValueTF {#1}
{\mathbf{1}  }
{\mathbf{1}\left\{ {#1} \right\} }%
}

\newcommand{\lawequals}{\overset{\mathscr{L}}{=}}
\DeclareDocumentCommand{\Prto} {o} {
\IfNoValueTF {#1}
 {\overset{\Pr}{\longrightarrow}}
 { \xrightarrow[ #1 \to \infty]{\Pr }}
}
\DeclareDocumentCommand{\Asto} {o} {
\IfNoValueTF {#1}
 {\overset{\operatorname{a.s.}}{\longrightarrow}}
 {
 \xrightarrow[ #1 \to \infty]{\operatorname{a.s.} }
 }
}
\DeclareDocumentCommand{\Mgfto} {o} {
\IfNoValueTF {#1}
{\overset{\operatorname{mgf}}{\longrightarrow}}
{ \xrightarrow[ #1 \to \infty]{\operatorname{mgf} }}
}

\DeclareDocumentCommand{\Wkto} {o} {
\IfNoValueTF {#1}
 {\overset{(d)}{\longrightarrow}}
 { \xrightarrow[ #1 \to \infty]{(d) }}
}

\newcommand{\drift}{\mathfrak{f}}

\begin{document}

 \title{The maximum deviation of the $\text{Sine}_\beta$ counting process}
\author{Diane Holcomb and Elliot Paquette}

\maketitle

\begin{abstract}
In this paper, we consider the maximum of the 
$\text{Sine}_\beta$ counting process from its expectation.  
We show the leading order behavior is consistent with the predictions of 
log--correlated Gaussian fields, also consistent with
work on the imaginary part of the log--characteristic polynomial of random 
matrices.
We do this by a direct analysis of the stochastic sine equation, which gives 
a description of the continuum limit of the Pr\"ufer phases of a Gaussian $\beta$--ensemble matrix.  
\end{abstract}

The $\text{Sine}_\beta$ point process (\cite{ValkoVirag}), which arises as the local point process limit of the eigenvalues of $\beta$--ensembles, can be defined in terms of the SDE 
\begin{equation}
  \label{eq:sde}
  d\alpha_{x,t}= x \frac{\beta}{4} e^{- \frac{\beta}{4} t} dt+ \re \left[ \left( e^{-i \alpha_{x,t}}-1\right)dZ_t\right], \qquad \alpha_{x,0}=0.
\end{equation}
Specifically, sending $t \to \infty,$ $\alpha_{x,t}/(2\pi)$ converges for all $x$ to an integer valued limit, which is the counting function of the $\text{Sine}_\beta$ point process.    

We are interested in the question of whether this function is an example of a process that should satisfy log--correlated field predictions. For an overview on work related to log--correlated Gaussian and approximately Gaussian processes see \cite{ArguinSurvey, ZeitouniSurvey}. This question follows naturally from the fact that the counting function of Sine$_\beta$ is a scaling limit of the imaginary part of the logarithm of the characteristic polynomial of random matrices. Such Gaussian log--correlated field predictions have been proven for a variety of matrix models \cite{ABB, PaquetteZeitouni02,CMN, LambertPaquette01}. Similar work has been done for randomized models of the Riemann $\zeta$ function \cite{ABH}, and also for the $\zeta$ function itself \cite{ABBRS, Najnudel}. For further discussion of the connections between the $\zeta$ function and random matrix theory see \cite{KeatingSnaith}.

We consider the process 
$
N(x)=\lim_{t\to \infty} \frac{\alpha_{x,t}-\alpha_{-x,t}}{2\pi},
$
which counts the number of points in the $\text{Sine}_\beta$ point process between $[-x,x]$ for any $x > 0.$  This process exhibits a purer analogy with log--correlated fields (see Remark \ref{rem:otherN} for details). 
We show that:
\begin{theorem}
  \label{thm:goal}
  \[
 \frac{\max_{0\le \lambda\le x}[N(\lambda)- \frac{\lambda}{\pi}]}{\log x} \Prto[x]
  \frac{2}{\sqrt \beta \pi}.
\]
\end{theorem}
\noindent Moreover, we do this by a direct argument for the $\text{Sine}_\beta$ process. Another possible approach might be to use the recent \cite{ValkoVirag3}, which gives a coupling between the $\text{Sine}_\beta$ and C$\beta$E point processes, to transfer estimates from the random matrix process to the continuum limit.

Observe that as the process $N(\lambda)$ is almost surely non--decreasing, we may immediately replace this maximum over all $0 \leq \lambda \leq x$ by the maximum over any discrete net of $[0,x]$ with maximum spacing $o(\log x).$  Likewise, we may assume that $x$ is an integer.  Going forward, we will take $\lambda$ and $x$ to be integers. The monotonicity of $N(\lambda)$ may be seen from the SDE description by observing that the noise term vanishes at multiples of $2\pi$ and the drift is positive for $\lambda>0$ and negative for $\lambda<0$ (\cite[Proposition 9(ii)]{ValkoVirag}). 

It should be noted there is another SDE description due to \cite{KillipStoiciu} (only recently proven to give rise to the same process by \cite{Nakano}, while another proof follows from \cite{ValkoVirag2}), which can be related to \eqref{eq:sde} by a time--reversal.  This arises due to an order reversal of the Pr\"ufer phases, for which reason the correlation structure is reversed from the previously studied C$\beta$E model. The processes $\alpha_{x,t}$ and $\alpha_{y,t}$ are strongly correlated for large times and weakly correlated for small times.  We elaborate upon the correlation structure in \eqref{eq:Mheuristic}.

\subsection*{Heuristic}

We will name the martingale part of $\alpha_{\lambda, t}-\alpha_{-\lambda,t}$ diffusion:
\begin{equation}
M_{\lambda,t}= \re \int_0^t (e^{-i \alpha_{\lambda,s}}-e^{-i\alpha_{-\lambda,s}})dZ_s.
\end{equation}
As the process $\alpha_{x,t}$ converges for all $x \in \R$ when $t \to \infty,$ so does $M_{\lambda,t}$ converge for all $\lambda \in \R$ when $t \to \infty.$  Moreover, 
\[
  2\pi N(\lambda) - 2\lambda
  =
\re \int_0^\infty (e^{-i \alpha_{\lambda,s}}-e^{-i\alpha_{-\lambda,s}})dZ_s
=M_{\lambda,\infty}.
\]
Therefore we can reformulate Theorem \ref{thm:goal} as
\begin{equation}
  \label{eq:mglegoal}
  \frac{\max_{0\le \lambda\le x}M_{\lambda,\infty} }{\log x} \Prto[x] 
  \frac{4}{\sqrt \beta}
  .
\end{equation}

Let $T_\lambda = \frac4\beta \log \lambda.$  This is heuristically the length of time that $M_{\lambda,t}$ needs to evolve so that it is within bounded distance of its limit.  Specifically, the variables $M_{\lambda,\infty} - M_{\lambda,T_\lambda}$ have a uniform--in--$\lambda$ exponential tail bound:
\begin{proposition}
  There is a constant $C = C_\beta$ so that for all $\lambda,r \geq 0,$
  \[
    \Pr\left[ 
      M_{\lambda,\infty} - M_{\lambda,T_\lambda}
      \geq
      C+
      r
    \right]
    \leq e^{-r/C}.
  \]
  \label{prop:tail}
\end{proposition}
\noindent Using the monotonoicity of $N(\lambda),$ we can also show that:
\begin{proposition}
  \[
    \frac{\max_{0\le \lambda\le x} |M_{\lambda,\infty} - M_{\lambda,T_\lambda} | }{\log x} \Prto[x] 0.
\]
\label{prop:timeout}
\end{proposition}
\noindent Hence we need only consider the process $M_{\lambda,t}$ up to time $t=T_\lambda.$ We delay the proofs of these propositions to Section~\ref{sec:background}.

Another representation for $M_{\lambda,t}$ is given by, for all $t \geq 0$ 
\begin{align}
M_{\lambda,t}&= \re \int_0^t (e^{-\frac{i}{2}(\alpha_{\lambda,s}-\alpha_{-\lambda,s})}-e^{-\frac{i}{2}(\alpha_{-\lambda,s}-\alpha_{\lambda,s})})e^{-\frac{i}{2}(\alpha_{\lambda,s}+\alpha_{-\lambda,s})}dZ_s \nonumber \\
&=  \re \int_0^t (e^{-\frac{i}{2}(\alpha_{\lambda,s}-\alpha_{-\lambda,s})}-e^{-\frac{i}{2}(\alpha_{-\lambda,s}-\alpha_{\lambda,s})}) (dV^{(\lambda)}_s+ i dW^{(\lambda)}_s) \nonumber \\
& = \int_0^t 2\sin\left( \tfrac{ \alpha_{\lambda,s}-\alpha_{-\lambda,s}}{2}\right)d W^{(\lambda)}_s. \label{eq:Wlambda}
\end{align}
where $dV^{(\lambda)}_s+ i dW^{(\lambda)}_s=e^{-\frac{i}{2}(\alpha_{\lambda,s}+\alpha_{-\lambda,s})}dZ_s$ is a standard complex Brownian motion.

Hence, the bracket process is given by
\[
  [M_\lambda]_t = \int_0^t 4\sin\left( \tfrac{ \alpha_{\lambda,s}-\alpha_{-\lambda,s}}{2}\right)^2\,ds.
\]
Applying the trig identity $2\sin(x)^2 = 1-\cos(2x),$ and treating the oscillating the term as negligible, we can consider $[M_\lambda]_t \approx 2t,$ for $t \leq T_\lambda.$  This allows us to roughly consider $M_{\lambda,T_\lambda},$ for the purpose of moderate deviations, as a centered Gaussian of variance $2T_\lambda.$

As for the correlation structure, 
\begin{equation}
  \label{eq:bracket}
  \begin{aligned} 
  ~[ M_\lambda, M_\mu ]_t
  &=
  \re
  \int_0^t 
  (e^{-i \alpha_{\lambda,s}}-e^{-i\alpha_{-\lambda,s}})
  (e^{i \alpha_{\mu,s}}-e^{i\alpha_{-\mu,s}})
  \,ds
  \end{aligned}
\end{equation}
Approximating $\alpha_{\lambda,t}$ by its drift in the equation above, we are led to the heuristic that $M_\lambda$ and $M_\mu$ behave approximately independently for $t \leq \frac{4}{\beta}\log_+|\lambda-\mu|$ and are maximally correlated for larger $t.$  
This leads to the cross variation heuristic:
\begin{equation}
  \label{eq:Mheuristic}
  [ M_\lambda, M_\mu ]_{T_\lambda \wedge T_\mu}
  \approx 2(T_\lambda \wedge T_\mu - \tfrac{4}{\beta}\log_+|\lambda-\mu|).
\end{equation}

We can define a Gaussian process that has the exact correlation structure suggested by the heuristics 
in \eqref{eq:Mheuristic}:
\begin{equation}
  G_{\lambda,t} = \re \int_0^t (e^{-i \Exp\alpha_{\lambda,s}}-e^{-i\Exp\alpha_{-\lambda,s}})dZ_s.
  \label{eq:Gdef}
\end{equation}
For this process, we have correlation given by
\[
  [G_\lambda,G_\mu]_t = 4\int_0^t 
  \sin\left( \lambda (1-e^{-\tfrac \beta 4 s})\right)
  \sin\left( \mu (1-e^{-\tfrac \beta 4 s})\right)
  \,ds.
\]
On the supposition that  the maximum of $\lambda \mapsto M_{\lambda,\infty}$ is well modeled by the maximum of the field $(G_{\lambda, T_{\lambda}}, 0 \leq \lambda \leq x)$, we are led to the following conjecture.
\begin{conjecture}
  \label{conj:actual}
  There is a random variable $\xi$ so that
  \[
      \max_{0\le \lambda\le x}(M_{\lambda,\infty})
      -\tfrac{4}{\sqrt{\beta}}\left( \log x  - \tfrac34 \log\log x\right)
      \Wkto[x] \xi.
  \] 
\end{conjecture}
\noindent Indeed by a theorem of \cite{DRZ}, full convergence could be proven for the $G_{\lambda,T_{\lambda}}$ field. One might expect that the distribution of $\xi$ is sensitive to the model and so could be different than in the Gaussian case.

\begin{remark} \label{rem:otherN}
  If we instead considered the one--sided problem, we would instead see
  \[
    \frac{
      \max_{0 \leq \lambda \leq x} \left[ \alpha_{\lambda,\infty} - \lambda \right]
    }{
    \log x}
    \Prto[x]
    \frac{4}{\sqrt{2\beta}}.
  \]
  We would be led to considering the martingale
  \[
    V_{\lambda,t}= \re \int_0^t (e^{-i \alpha_{\lambda,s}}-1)dZ_s.
  \]
  which has quadratic variation $[V_\lambda]_t\approx 2t$ for $t< T_\lambda$ and cross variation: 
  \begin{equation}
    \label{eq:newbracket}
    \begin{aligned} 
      ~[ V_\lambda, V_\mu ]_{T_\lambda \wedge T_\mu}
      &=
      \re
      \int_0^t 
      (e^{-i \alpha_{\lambda,s}}-1)
      (e^{i \alpha_{\mu,s}}-1)
      \,ds
      \approx T_\lambda \wedge T_\mu  + \frac12 [ M_\lambda, M_\mu ]_{T_\lambda \wedge T_\mu}.
    \end{aligned}
  \end{equation}
  Thus, the process has an additional positive correlation, which is heuristically equivalent to adding a common standard normal of variance $\tfrac{4}{\beta}\log x$ to every $V_{\lambda,\infty}$ for $\delta x \leq \lambda \leq x.$  In particular this is too small to change the behavior of the maximum.  As working with $V_{\lambda,t}$ does not materially change the argument, we have not pursued it here.
\end{remark}


\section{Background tools}
\label{sec:background}
We begin with the proofs of Propositions \ref{prop:tail} and \ref{prop:timeout}.
These rely heavily on basic properties of the diffusion established in \cite[Proposition 9]{ValkoVirag}.
\subsection*{ Delayed proofs from introduction }
\begin{proof}[Proof of Proposition \ref{prop:tail}]
  Observe first by integrating the drift
  \begin{equation}
    \label{eq:subtract1}
    M_{\lambda,\infty} - M_{\lambda,T_\lambda}
    =\alpha_{\lambda,\infty} - \alpha_{\lambda,T_\lambda}
    -1.
  \end{equation}
  Consider the process $v$ that satisies
  \[
    dv_{t}= \lambda \frac{\beta}{4} e^{- \frac{\beta}{4} t}\one[t \leq T_\lambda] dt+ \re \left[ \left( e^{-i v_{t}}-1\right)dZ_t\right], \qquad v_{0}=0.
  \]
  Then $\alpha_{\lambda,t}$ and $v_t$ are equal until $T_\lambda.$  After this time, $v$ never crosses another multiple of $2\pi.$  Moreover, it eventually converges to a multiple of $2\pi$ (\cite[Proposition 9(iv)]{ValkoVirag}).  Hence we have
  \begin{equation}
    \label{eq:ua}
    |v_\infty - \alpha_{\lambda,T_{\lambda}}| \leq 2\pi.
  \end{equation}
  On the other hand $\alpha_{\lambda,\infty}-v_\infty$ has the same law as $\alpha_{1,\infty}$.
  By \cite[Proposition 9(viii)]{ValkoVirag}, this has an exponential tail bound.
\end{proof}

\begin{proof}[Proof of Proposition~\ref{prop:timeout}]
  By \eqref{eq:subtract1}, it suffices to show the same for $\alpha_{\lambda,\infty} - \alpha_{\lambda,T_\lambda}.$  The diffusion $\alpha_{\lambda,t}$ can not cross below an integer multiple of $2\pi.$  Hence if $s \leq t,$ for all $\lambda \geq 0$
  \(
   \alpha_{\lambda,s} \leq \alpha_{\lambda,t} + 2\pi.
  \)
  This implies 
  \[
    \min_{0 < \lambda \leq x} \left(\alpha_{\lambda,\infty} - \alpha_{\lambda,T_\lambda}\right) \geq -2\pi,
  \]
  and it suffices to consider an upper bound.
For $x/2 \leq \lambda \leq x,$ we can estimate
\[
  \alpha_{\lambda,\infty} - \alpha_{\lambda,T_\lambda}
  \leq
  \alpha_{\lambda,\infty} - \alpha_{\lambda,T_{x/2}} + 2\pi
\]
Let $v_\lambda$ satisfy
\[
  dv_{\lambda,t}= \lambda \frac{\beta}{4} e^{- \frac{\beta}{4} t}\one[t \leq T_{x/2}] dt+ \re \left[ \left( e^{-i v_{\lambda,t}}-1\right)dZ_t\right], \qquad v_{\lambda,0}=0.
\]
As $v_\lambda$ can not cross multiples of $2\pi,$ for any $\lambda\in \R,$ after $T_{x/2},$ we have
\[
  \alpha_{\lambda,\infty} - \alpha_{\lambda,T_{x/2}} + 2\pi
  \leq
  \alpha_{\lambda,\infty} - v_{\lambda,\infty} + 4\pi.
\]
On the other hand $\alpha_{\lambda,t} - v_{\lambda,t}$ is monotone increasing in $\lambda$ almost surely (as the difference for parameters $\lambda_1>\lambda_2$ satisfies an SDE that can not cross below $0$, c.f.\ \cite[Proposition 9(ii)]{ValkoVirag}).  
Combining the work so far, we have the bound
\[
  \max_{x/2 \leq \lambda \leq x}
  \left(
  \alpha_{\lambda,\infty} - \alpha_{\lambda,T_\lambda}
  \right)
  \leq
  \alpha_{x,\infty} - v_{x,\infty} + 4\pi.
\]
Using the equality in law given by
\[
  \left(
  \alpha_{x,t+T_{x/2}} - v_{x,t+T_{x/2}}, t \geq 0
  \right) 
  \lawequals 
  \left(\alpha_{2,t}, t \geq 0\right),
\]
and by \cite[Proposition 9(viii)]{ValkoVirag}, $\alpha_{2,\infty}$ has an exponential tail bound depending only on $\beta.$  Applying the same argument for $j\in \N$ and $x 2^{-j-1} \leq \lambda \leq x 2^{-j},$ we may use a union bound up to $j$ on the order of $\log x$ to conclude that there is a constant $C_\beta$ so that
\begin{equation}\label{eq:sharper}
    \max_{0 < \lambda \leq x} \left(\alpha_{\lambda,\infty} - \alpha_{\lambda,T_\lambda}\right) \leq C_\beta\log\log x
\end{equation}
with probability going to $1$ as $x \to \infty.$
\end{proof}

\subsection*{Oscillatory integrals}

For each $\lambda \in \R,$ suppose that
$A_{\lambda,t}$ is an adapted finite variation process so that $|A_{\lambda,t}| \leq \xi \in (0,\infty)$ for all time almost surely 
and suppose that $X_{\lambda,t}$ is a martingale satisfying $d[X_{\lambda}]_t \leq 2.$
Suppose that
\begin{equation}
  \label{eq:gusde}
  du_{\lambda, t}
  =
  \lambda \tfrac{\beta}{4}e^{-\tfrac \beta 4 t}dt
  +A_{\lambda,t} dt
  +dX_{\lambda,t},\quad\quad u_{\lambda,0} = 0.
\end{equation}

\begin{proposition}
\label{prop:oscbnd}
Let $u_{\lambda,t}$ satisfy \eqref{eq:gusde} and let $\drift (t) = \frac{\beta}{4}e^{-\frac{\beta}{4}t}$, then for each fixed $\beta>0$ there exist constants $R$ and $\gamma$ uniform in $T$ and $\lambda, a\in \R$ such that 
\begin{align}
  \Exp\left[\sup_{0 \le t \le T} \left|\int_0^t e^{i a u_{\lambda,s}} ds\right| \right] \le \frac{R(1+|\xi|)}{|a\lambda| \drift(T)},
\end{align}
and for all $C>0$
\begin{align}
\Pr\left(\sup_{0\le t \le T}\left|\int_0^t e^{i a u_{\lambda,s}}ds\right|- \frac{R(1+|\xi|)}{|a\lambda| \drift(T)} \ge C \right) \le \exp \left[ - \gamma C^2 a^2 \lambda^2 \drift(T)^2\right].
\end{align}
\end{proposition}

\begin{proof}
    The theorem is vacuous if $a\lambda = 0,$ so we may assume this is not the case.
  Writing $u_t$ in its integrated form, we have
  \[
    \begin{aligned}
     u_t &= \lambda\left(1-\frac{4}{\beta}\drift(t)\right) + \mathcal{R}_t \\
\mathcal{R}_t
&= \int_0^t 
\left\{ A_{\lambda,s}ds + dX_{\lambda,s} \right\}
.
  \end{aligned}
  \]
Let $H(t) = 1-\frac{4}{\beta}\drift(t)$ and $\Lambda(t) = \int_0^t e^{i a \lambda H(s)}ds$ , then we may use It\^o integration by parts to get  
\begin{align}
  \label{eq:ibp}
\int_0^t e^{i a\lambda u_s} ds 
= \int_0^t e^{i a \lambda H(s)} e^{i a\mathcal{R}_s} ds 
= e^{i a\mathcal{R}_t}\Lambda(t)
- \int_0^t \Lambda(s) i a e^{i a\mathcal{R}_s}d\mathcal{R}_s 
+ \frac{a^2}{2}\int_0^t \Lambda(s) e^{i a\mathcal{R}_s} d[\mathcal{R}]_s.
 \end{align}
 Now observe that $\Lambda(t)$ may be bounded in the following way:
\begin{align*}
\int_0^{t} e^{i a\lambda  H(s)}ds  
&= \int_0^{t} \frac{1}{i a\lambda \drift(s)} \frac{d}{ds} e^{i a \lambda  H(s)}ds
= \frac{ 4 e^{\frac{\beta}{4}t}}{\beta i a \lambda}
\left\{e^{i a\lambda  H(t)} - 1 \right\}
- \frac{1}{i a\lambda } \int_0^{t} e^{\frac{\beta}{4}s}  
\left\{e^{i a \lambda  H(s)}-1\right\}ds.
\end{align*}
This gives us $|\Lambda(s)| \le \frac{16}{\beta |a \lambda|} e^{\frac{\beta}{4}t}$. Applying this to our integrated equation we get for the finite variation terms 
\[
\left| \int_0^t \Lambda(s)e^{i \mathcal{R}_s} 
 aA_{\lambda,s} ds +  
\frac{a^2}{2}\int_0^t \Lambda(s)e^{i a \mathcal{R}_s} d[\mathcal{R}]_s   \right| 
\le 
\frac{16}{\beta a \lambda}
e^{\frac{\beta}{4}t}
\left( |a|\xi + a^2 \right).
\]

By \eqref{eq:ibp} and the triangle inequality, it remains to show the desired tail bound and supremum bound for the martingale $V_t$ given by
\[
  V_t =  \int_0^t \Lambda(s) i a e^{i a\mathcal{R}_s}\cdot dX_{\lambda,s} 
\]
Note we have an easy bracket bound, for $\sigma \in \left\{ 1, i \right\}$ given by
\[
[\Re (\sigma V)]_t \leq \int_0^t 2\Lambda(s) a^2\,ds \leq \frac{C_\beta}{ \lambda ^2} |a| e^{\tfrac{\beta}{2} t}
\]
for some constant $C_\beta.$
Hence the desired bounds follow immediately from the Dambis--Dubins--Schwarz theorem (\cite[Theorem V.1.6]{RevuzYor} or \cite[Theorem II.42]{Protter}) and Doob's inequality.
\end{proof}

\subsection*{Tilting}
We now want to look at the measure tilted so that $W^{(\lambda)}$ (see \eqref{eq:Wlambda}) has a drift. In particular 
for deterministic $\eta \in \R,$
we consider the measure $Q_{\xi,\lambda}$ so that  
\[
  dX_s = dW^{(\lambda)}_s - \eta \sin\left( \tfrac{ \alpha_{\lambda,s}-\alpha_{-\lambda,s}}{2}\right) ds
\]
is a standard Brownian motion up to time $T$ under $Q_{\eta,\lambda}$. By Girsanov (see e.g.\ \cite[Theorem III.8.46]{Protter}) we get that 
\begin{equation}
  \label{eq:girsanov}
  \frac{dQ_{\eta,\lambda}}{d\Pr} = 
  \mathcal{E}(\eta M_\lambda)
  =
  \exp( \eta M_{\lambda,T} - \tfrac{\eta^2}{2}[M_\lambda]_T)
\end{equation}
Since $\sin^2(x) \leq 1$ we have that the bracket process of $[M_\lambda]_t \leq T$ almost surely for all $t \geq 0$.  In particular, the exponential martingale is uniformly integrable by Novikov's condition for all $\eta \in \R$.

Under $Q_{\eta, \lambda}$ the law of $\alpha_{\lambda,t} - \alpha_{-\lambda,t}$ changes; it can be succinctly described as the solution to
\begin{equation}
  \label{eq:usde}
  du_{\lambda,\eta,t}
  =
  2\lambda \tfrac{\beta}{4}e^{-\tfrac \beta 4 t}dt
  +2\eta\sin\left( \tfrac{ u_{\lambda,\eta,t}}{2}\right)^2dt
  +2\sin\left(\tfrac{u_{\lambda,\eta,t}}{2}\right)dX_t,\quad\quad u_0 = 0
\end{equation}
for a Brownian motion $dX,$ which we call the \emph{accelerated stochastic sine equation} with acceleration $\eta$.  Let $M_{\lambda,\eta,t}$ be the martingale part of $u_{\lambda,\eta,t}.$

\subsection*{Martingale bounds}

Using the Girsanov transformation, we now give a nearly sharp tail bound for $M_\lambda.$
\begin{proposition}
  For any $\eta \in \R,$ there is an $R > 0$ so that for all $\lambda > 0,$  all $T \leq T_\lambda$
\[
\Pr\left(\sup_{0\le t \le T} M_{\lambda,\eta,t} \ge C\right) 
\le \exp\left[\frac{-C^2}{4(T+R)}\left( 1- \frac{C^2R}{2(T+R)^3} \right) \wedge \frac{-C^{4/3}}{4T^{1/3}} \right].
\]
and
\[
\Pr\left(\inf_{0\le t \le T} M_{\lambda,\eta,t} \le -C\right) 
\le \exp\left[-\frac{C^2}{4(T+R)} \left( 1- \frac{C^2R}{2(T+R)^3} \right) \wedge \frac{-C^{4/3}}{4T^{1/3}}\right]
\]
\label{prop:mgle}
\end{proposition}
\begin{remark}
  For $C$ up to the order of magnitude of $T^{3/2}$ the Gaussian tail majorizes the martingale tail.  For larger $C,$ the second term majorizes the martingale tail.  For much much larger $C$ (on the order $T^2$) a small change in the proof gives decay of order $e^{-cC^{4/3}}.$  A large deviations principle for $N_{\lambda}$ is proven in \cite{DHBV} which suggests a stronger tail bound ought to be true.
\end{remark}

\begin{proof}

  Let $X_t$ be a standard Brownian motion, and let $w$ solve \eqref{eq:usde} the accelerated stochastic sine equation with acceleration $\eta.$  Let $M$ be the martingale part of $w.$
Let $\xi \in \R,$
and apply Doob's inequality to the submartingale $e^{\xi M_{t}}$ to get 
\[
  \Pr\left( \sup_{0\le t \le T} M_{t} \ge C\right) \le e^{-\xi C} \Exp(e^{\xi M_{T}}).
\] 
Applying \eqref{eq:girsanov}, we have that
\[
 \Exp(e^{\xi M_{T}})
 =
 \Exp\left(\mathcal{E}(\xi M_{T}) e^{\tfrac{\xi^2}{2} [M]_T}\right)
 =
 \hat Q_E\left(e^{\tfrac{\xi^2}{2} [M]_T}\right),
\]
with $\hat Q_E(\cdot)$ the expectation under the probability measure $\hat Q$ defined by
\[
  \frac{d\hat Q}{d \Pr}
  = \mathcal{E}(\xi M_{T}).
\]
By the Girsanov theorem,
\[
  dY_s = dX_s - 
  \xi
  \sin\left( \tfrac{ w_t}{2} \right)\,ds
\]
is a $\hat Q$--Brownian motion.  Hence, 
\[
M_{t} 
= \int_0^t 2\sin\left( \tfrac{w_s}{2}\right)dY_s
+ \int_0^t 2\xi\sin\left( \tfrac{w_s}{2}\right)^2 ds.
\]
Further, the law of $w_s$ changes under $\hat Q,$ as we have that
\[
  dw_{t}
  =
  2\lambda \tfrac{\beta}{4}e^{-\tfrac \beta 4 t}dt
  +2(\xi+\eta)\sin\left( \tfrac{ w_t}{2}\right)^2dt
  +2\sin\left(\tfrac{w_t}{2}\right)dY_t,\quad\quad w_0 = 0.
\]
Hence, under $\hat Q,$ $w$ is a solution of the accelerated stochastic sine equation with acceleration $\xi+\eta.$ 

As for the bracket, we have that for $t \leq T$
\[
  [M_\lambda]_t
  = 
  \int_0^t 4\sin\left( \tfrac{w_s}{2}\right)^2\,ds
  = 
  2t - 
  \int_0^t 2\cos\left( {w_s}\right)\,ds.
\]
Using Proposition~\ref{prop:oscbnd}, we have that for $T \leq T_\lambda,$ there is an $R$ independent of $\xi$ and $\eta$ so that for all $C>0$
\[
  \hat Q\left(\int_0^T -2\cos\left( {w_s}\right)\,ds \geq R(1+|\xi+\eta|) + C\right) \leq e^{-C^2/R}.
\]
Therefore, we have that for $T \leq T_\lambda$ 
\[
 \hat Q_E\left(e^{\tfrac{\xi^2}{2} [M_{\lambda}]_T}\right)
 =e^{\xi^2 T}\hat Q_E\left( \exp\left(\int_0^T -\xi^2\cos\left( {w_s}\right)\,ds\right)\right)
 \leq 
 e^{\xi^2(T+S) + S\xi^4}
\]
for some constant $S>0$ independent of $\xi,\lambda$ or $T$ but depending on $\eta.$  

There remains to optimize in $\xi.$  From the work so far, we have
\[
  \Pr\left( \sup_{0\le t \le T} M_{t} \ge C\right) 
  \le e^{-\xi C} \Exp(e^{\xi M_{T}})
  \leq
 e^{-\xi C + \xi^2(T+S) + S|\xi|^3}.
\]
Taking $\xi = \tfrac{C}{2(T+S)},$
\[
  \Pr\left( \sup_{0\le t \le T} M_{t} \ge C\right) 
  \leq
  \exp\left[
    -\frac{C^2}{4(T+S)} + \frac{S C^4}{8(T+S)^4}
  \right],
\]
and taking $\xi = (C/(4T+4S))^{1/3}$ gives
\[
  \Pr\left( \sup_{0\le t \le T} M_{t} \ge C\right) 
  \leq
  \exp\left[
    -\frac{3C^{4/3}}{4(4(T+S))^{1/3}} + \frac{C^{2/3}(T+S)^{1/3}}{4^{2/3}}
  \right].
\]
Hence the desired bound holds
by taking the second bound for $C > P(T+S)$ and $P$ sufficiently large, and the first 
bound for $C \leq P(T+S).$

The statement about the infimum may be proved in an identical fashion by reformulating it as an equivalent bound on the supremum of $-M_\lambda$. We would then use the submartingale $e^{-\xi M_\lambda}$ and use $[M_\lambda]_t = [-M_\lambda]_t$. 

\end{proof}
\section{Main theorem}
\label{sec:main}
\subsection*{The one--point upper bound}

Using Proposition~\ref{prop:mgle} with $\eta=0$, we can give the upper bound in \eqref{eq:mglegoal}.  
\begin{proposition}
  For any $\delta > 0$
  \begin{equation*}
    \lim_{x \to \infty}
    \Pr\left(
    {\max_{0\le \lambda\le x}M_{\lambda,T_\lambda} } >  \left(\frac{4}{\sqrt \beta} + \delta\right){\log x}  
    \right) = 0
\end{equation*}
\end{proposition}
\begin{proof}
  As commented, it suffices to bound the probability for natural numbers $\lambda$ and $x.$
  By Proposition~\ref{prop:mgle} for any $\delta>0$ sufficiently small
  there is an $\epsilon > 0$ 
  and an $x_0$ sufficiently large so that for all $x > x_0$
  and all $x > \lambda > \exp( (\log x)^{3/4})$ 
  \[
    \Pr\left(
    M_{\lambda,T_\lambda} >  \left(\frac{4}{\sqrt \beta} + \delta\right){\log x}  
    \right)
    \leq \exp\left(-(\log x)^2\tfrac{\left(\frac{4}{\sqrt \beta} + 2\delta\right)^2}{\tfrac{16}{\beta}\log \lambda}\right)
    \leq \exp( -(\log x)(1+\epsilon)).
  \]
  For smaller $\lambda,$ we have, taking the $4/3$--power bound in Proposition~\ref{prop:mgle}, that for some $C_{\beta,\delta}$
  \[
    \Pr\left(
    M_{\lambda,T_\lambda} >  \left(\frac{4}{\sqrt \beta} + \delta\right){\log x}  
    \right)
    \leq \exp\left(-(\log x)^{13/12}C_{\beta,\delta}
    \right)
  \]
  Hence, taking a union bound over all natural numbers $\lambda$ less than $x$ gives the desired bound.
\end{proof}
\begin{remark}
  In fact, the proof is easily modified to give 
  \[
    \limsup_{\lambda \to \infty}
    \left(
    \frac{M_{\lambda,T_\lambda} }
    { \log \lambda}
    \right)
    \leq \frac{4}{\sqrt{\beta}}, \quad \text{a.s.}
  \]
\end{remark}

\subsection*{The tube event and the lower bound}

Let $x$ be a natural number,
and let $R$ be a large parameter to be chosen later.
Let $T_\lambda' = T_\lambda - R^2\sqrt{\log\lambda}.$
Define an event $\mathcal{A}_\lambda$ given by 
\[
  \mathcal{A}_\lambda
  =\left\{ 
    \begin{aligned}
      &|M_{\lambda,t} - \sqrt{\beta}t| \leq R \sqrt{\log x}, ~\forall~0 \leq t \leq T_x'; \\ 
    &|[M_{\lambda}]_t - 2t| \leq R, ~\forall~0 \leq t \leq T_x'  
  \end{aligned}
  \right\}.
\]

Let $x$ be a natural number, and define 
\begin{equation}
  S_x = \sum_{\lambda=x}^{2x} \mathcal{E}( \sqrt{\beta} M_{\lambda,T_x'}) \one[ \mathcal{A}_\lambda]
  \label{eq:Sx}
\end{equation}
Notice that with this definition of $S_x$ we will have that $S_x>0$ if and only if the event $\mathcal{A}_\lambda$ occurs for some integer $\lambda \le x$. Using the Cauchy-Schwarz inequality for non-negative random variables, we arrive at the Paley-Zygmund inequality
\begin{equation}
\label{eq:PZinequality}
  \Pr( S_x > 0)\Exp S_x^2 \geq (\Exp S_x)^2.
\end{equation}
We wish to show that this has probability going to $1$ as $\lambda \to \infty$ for any $\delta > 0.$
Hence, we need to produce a lower bound of the form 
\[
  \Exp[\mathcal{E}( \sqrt{\beta} M_{\lambda,T_x'}) \one[ \mathcal{A}_\lambda]]
  =Q_{\sqrt{\beta},\lambda}(\mathcal{A}_\lambda) \ge  1 - C_\beta e^{-R^{4/3}/C_\beta}
  \label{eq:onepoint},
\]
and we need to produce a similar upper bound on
\[
  \Exp[
    \mathcal{E}( \sqrt{\beta} M_{\lambda_1,T_x'}) \one[ \mathcal{A}_{\lambda_1}]
    \mathcal{E}( \sqrt{\beta} M_{\lambda_2,T_x'}) \one[ \mathcal{A}_{\lambda_2}]
  ].
\]
From these bounds we will be able to show that as $x\to \infty$
\begin{equation}
  (\Exp S_x^2)/x^2 \to 0 \qquad \text{ and } \qquad \Exp S_x \geq x(1 - C_\beta e^{-R^{4/3}/C_\beta}).
  \label{eq:Sx2o1}
\end{equation}
Hence, we conclude \eqref{eq:PZinequality} that for any $\epsilon >0$ there is an $R$ sufficiently large and an $x_0$ sufficiently large so that for all $x > x_0$ 
\[
  \Pr( S_x > 0) \geq \frac{ (\Exp S_x)^2}{ \Exp S_x^2} \geq 1-\epsilon.
\]
We have therefore shown that by letting $R_x$ tend arbitrarily slowly to infinity
\begin{equation}
  \label{eq:2nd}
  \max_{x \leq \lambda \leq 2x} \left\{
    M_{\lambda,T_x'}
  \right\}
  \geq \sqrt{\beta}T_x' - R_x\sqrt{\log x},
\end{equation}
with probability going to $1$ as $x \to \infty.$

\subsection*{One point lower bound}
We need to find a lower bound on  
\[
  \Exp[\mathcal{E}( \sqrt{\beta} M_{\lambda,T_x'}) \one[ \mathcal{A}_\lambda]]
  =Q_{\sqrt{\beta},\lambda}(\mathcal{A}_\lambda),
\]
which is on the order of unity.  Recall that under $Q_{\sqrt{\beta},\lambda}$ the process $\alpha_{\lambda,\cdot} - \alpha_{-\lambda,\cdot}$ follows the accelerated stochastic sine equation \eqref{eq:usde} with $\xi=\sqrt{\beta}.$  The process $M_{\lambda,t}$ referenced in the event $\mathcal{A}_\lambda$ can be expressed as 
\[
  M_{\lambda,t} = u_{\lambda,\xi,t} - 2\lambda(1- \tfrac4\beta\drift(t)).
\]
Meanwhile, the performing the Doob decomposition on $u_{\lambda,\xi,t},$ we have
\[
  M_{\lambda,\xi,t} = u_{\lambda,\xi,t} 
  - 2\lambda(1- \tfrac4\beta\drift(t))
  - \int_0^t 2\xi\sin\left( \tfrac{ u_{\lambda,\xi,s}}{2}\right)^2ds
\]
The bracket process $[M_{\lambda,\xi}]_t$ is given as before by
\[
  [M_{\lambda,\xi}]_t
  = 
  \int_0^t 4\sin\left( \tfrac{u_{\lambda,\xi,s}}{2}\right)^2\,ds
  = 
  2t - 
  \int_0^t 2\cos\left( {u_{\lambda,\xi,s}}\right)\,ds.
\]
Hence we can write
\[
  \begin{aligned}
  Q_{\xi,\lambda}(\mathcal{A}_\lambda)
  \geq
  1
  -
  &Q_{\xi,\lambda}
  \left( 
  \sup_{0 \leq t \leq T_x'} \left|M_{\lambda,\xi,t} + \int_0^t \xi \cos(u_{\lambda,\xi,s})\,ds\right| > R\sqrt{\log x}
  \right) 
  \\
  -&Q_{\xi,\lambda}
  \left( 
  \sup_{0 \leq t \leq T_x'} 
  \left|\int_0^t 2\cos\left( {u_{\lambda,\xi,s}}\right)\,ds\right| > R
  \right). 
\end{aligned}
\]
By Propositions \ref{prop:oscbnd} and \ref{prop:mgle}, we conclude that
\begin{equation}
  Q_{\xi,\lambda}(\mathcal{A}_\lambda)
  \geq 1 - C_\beta e^{-R^{4/3}/C_\beta}
  \label{eq:onepoint}
\end{equation}
for some $C_\beta$ sufficiently large and all $\lambda$ sufficiently large.

\subsection*{Two point bound}

Following the heuristic \eqref{eq:Mheuristic}, we treat $M_{\lambda_1,t}$ and $M_{\lambda_2,t}$ as uncorrelated until $T_* = \tfrac4\beta \log_+ | \lambda_1 - \lambda_2 |.$ 
Without loss of generality, suppose that $\lambda_2 \geq \lambda_1.$
On the event $\mathcal{A}_{\lambda_2},$ we can estimate 
\[
  \begin{aligned}
    \mathcal{E}( \sqrt{\beta} M_{\lambda_2,T_x'})
    &
    =
    \mathcal{E}( \sqrt{\beta} M_{\lambda_2,T_*})
    \exp\left(
    \sqrt{\beta} ( M_{\lambda_2,T_x'} - M_{\lambda_2,T_*})
    -\tfrac{\beta}{2} ([M_{\lambda_2}]_{T_x'} - [M_{\lambda_2}]_{T_*})
    \right) \\
    &
    \leq
    \mathcal{E}( \sqrt{\beta} M_{\lambda_2,T_*})
    \exp\left(
    2\sqrt{\beta} R \sqrt{\log x} + {\beta} R
    \right).
  \end{aligned}
\]
Hence, we have the estimate
\begin{equation}
  \begin{aligned}
  &\Exp[
    \mathcal{E}( \sqrt{\beta} M_{\lambda_1,T_x'}) \one[ \mathcal{A}_{\lambda_1}]
    \mathcal{E}( \sqrt{\beta} M_{\lambda_2,T_x'}) \one[ \mathcal{A}_{\lambda_2}]
  ] \\
  \leq
  &\Exp\left[
    \mathcal{E}( \sqrt{\beta} M_{\lambda_1,T_x'}) 
    \mathcal{E}( \sqrt{\beta} M_{\lambda_2,T_*})
    \exp\left(
    2\sqrt{\beta} R \sqrt{\log x} + {\beta} R
    \right)
  \right]. 
\end{aligned}
  \label{eq:decorrelationpoint}
\end{equation}

We now observe that
\begin{equation}
    \mathcal{E}( \sqrt{\beta} M_{\lambda_1,T_x'}) 
    \mathcal{E}( \sqrt{\beta} M_{\lambda_2,T_*})
    =
    \mathcal{E}( \sqrt{\beta} (M_{\lambda_1,T_x'} + M_{\lambda_2,T_*}))
    \exp\left( 
    \beta
    [M_{\lambda_1},M_{\lambda_2}]_{T_* \wedge T_x'}
    \right).
    \label{eq:cross}
  \end{equation}
By the Girsanov theorem, under the measure $\mathbb{S}$ with Radon--Nikodym derivative
\[
  \frac{d\mathbb{S}}{d\Pr} = \mathcal{E}( \sqrt{\beta} (M_{\lambda_1,T_x'} + M_{\lambda_2,T_*})),
\]
we have that there is a finite variation process $A_t$ bounded almost surely by an absolute constant so that
\[
  dU_t = dZ_t - \sqrt{\beta}A_t\,dt
\]
is a standard complex $\mathbb{S}$--Brownian motion. Here $Z_t$ is the standard complex Brownian motion used in equation (\ref{eq:sde}) under the measure $\mathbb{P}$. Meanwhile \eqref{eq:sde} (also c.f.\ \eqref{eq:bracket}) shows that  
\(
[M_{\lambda_1},M_{\lambda_2}]_t
\)
is a sum of integrals of $e^{i\sigma_1(\sigma_1\alpha_{\sigma_2 \lambda_1,t} + \sigma_3 \alpha_{\sigma_4 \lambda_2,t})}$ with $\sigma_j \in \left\{ 1,-1 \right\}.$ 
Applying Proposition~\ref{prop:oscbnd} to each of these integrals, we can conclude 
\[
  \Pr\left(
  [M_{\lambda_1},M_{\lambda_2}]_{T_* \wedge T_x'} > t + C
  \right)
  \leq e^{-t^2/C}
\]
for sufficiently large $C.$  Hence we conclude using \eqref{eq:cross} and \eqref{eq:decorrelationpoint} that there is some constant $C_\beta$ so that for any $R>0$
\begin{equation}
  \label{eq:crude}
  \Exp[
    \mathcal{E}( \sqrt{\beta} M_{\lambda_1,T_x'}) \one[ \mathcal{A}_{\lambda_1}]
    \mathcal{E}( \sqrt{\beta} M_{\lambda_2,T_x'}) \one[ \mathcal{A}_{\lambda_2}]
  ]
  \leq
  e^{C_\beta + 2R\sqrt{\beta \log x} + \beta R}.
\end{equation}

\subsection*{Fine estimate}

We also need an estimate that improves when $\lambda_1$ and $\lambda_2$ are well separated.  Once more, we estimate by dropping the indicators and writing
\begin{equation}
  \label{eq:finestart}
  \Exp[
    \mathcal{E}( \sqrt{\beta} M_{\lambda_1,T_x'}) \one[ \mathcal{A}_{\lambda_1}]
    \mathcal{E}( \sqrt{\beta} M_{\lambda_2,T_x'}) \one[ \mathcal{A}_{\lambda_2}]
  ]
  \leq
  \mathbb{S}
  \left( 
  \exp\left( 
    \beta
    [M_{\lambda_1},M_{\lambda_2}]_{T_x'}
    \right)
  \right),
\end{equation}
where 
\[
  \frac{d\mathbb{S}}{d\Pr} = \mathcal{E}( \sqrt{\beta} (M_{\lambda_1,T_x'} + M_{\lambda_2,T_x'})).
\]
Now, on applying Proposition~\ref{prop:oscbnd}, we have a tail bound of the form
\[
  \Pr\left(
  [M_{\lambda_1},M_{\lambda_2}]_{T_x'} > t + C_\beta/\Delta
  \right)
  \leq e^{-t^2\Delta^2/C_\beta}
\]
where $\Delta = |\lambda_1-\lambda_2|\drift(T_x')$ and $C_\beta > 0$ is a constant.  This leads to an estimate of the form
\begin{equation}
  \label{eq:fine}
  \Exp[
    \mathcal{E}( \sqrt{\beta} M_{\lambda_1,T_x'}) \one[ \mathcal{A}_{\lambda_1}]
    \mathcal{E}( \sqrt{\beta} M_{\lambda_2,T_x'}) \one[ \mathcal{A}_{\lambda_2}]
  ]
  \leq
  \exp\left( 
  C_\beta/\Delta
  \right).
\end{equation}
for some other $C_\beta$ and all $\Delta \geq 1.$ 

\subsection*{The second moment}

Here we estimate $\Exp S_x^2.$  
Recalling \eqref{eq:Sx}, we can write
\begin{equation}
  \Exp S_x^2 = 
  \sum_{\lambda_1=x}^{2x} 
  \sum_{\lambda_2=x}^{2x} 
  \Exp\left[
    \mathcal{E}( \sqrt{\beta} M_{\lambda_1,T_{\lambda_1}}) \one[ \mathcal{A}_{\lambda_1}]
      \mathcal{E}( \sqrt{\beta} M_{\lambda_2,T_{\lambda_2}}) \one[ \mathcal{A}_{\lambda_2}]
  \right].
  \label{eq:Sx2}
\end{equation}
We partition this sum according to the magnitude of $|\lambda_1-\lambda_2|.$  Let $S_0$ be all those pairs $(\lambda_1,\lambda_2)$ so that $|\lambda_1-\lambda_2| \geq x e^{- \tfrac12 R^2 \sqrt{\log x}}.$  Let $S_1$ be the remaining pairs.  Observe that the cardinality of $S_1$ is at most $2x^2e^{- \tfrac12 R^2 \sqrt{\log x}}.$


For terms in $S_0,$ we apply the fine bound \eqref{eq:fine}.  The term $\Delta$ that appears for such terms can be estimated uniformly by
\[
  \Delta \geq xe^{-\tfrac12 R^2 \sqrt{\log x}} \cdot \tfrac{\beta}{4} e^{-\log x + R^2 \sqrt{\log x}},
\]
which tends to $\infty$ with $x.$  In particular, we can estimate 
\begin{equation}
  \sum_{S_0} 
  \Exp\left[
    \mathcal{E}( \sqrt{\beta} M_{\lambda_1,T_x'}) \one[ \mathcal{A}_{\lambda_1}]
      \mathcal{E}( \sqrt{\beta} M_{\lambda_2,T_x'}) \one[ \mathcal{A}_{\lambda_2}]
  \right]
  \leq x^2\cdot (1+ O(e^{-\tfrac12 R^2\sqrt{\log x}})).
  \label{eq:Sx20}
\end{equation}

For the remaining terms, we apply the coarse bound \eqref{eq:crude}, using which we conclude that
\begin{equation}
  \sum_{S_1} 
  \Exp\left[
    \mathcal{E}( \sqrt{\beta} M_{\lambda_1,T_x'}) \one[ \mathcal{A}_{\lambda_1}]
      \mathcal{E}( \sqrt{\beta} M_{\lambda_2,T_x'}) \one[ \mathcal{A}_{\lambda_2}]
  \right]
  \leq x^2e^{C_\beta-\tfrac12 R^2\sqrt{\log x} + 2R\sqrt{\beta\log x}+\beta R}.
  \label{eq:Sx2j}
\end{equation}
Hence picking $R$ sufficiently large (anything larger than $4\sqrt{\beta}$ will do), we have combining \eqref{eq:Sx2}, \eqref{eq:Sx20} and \eqref{eq:Sx2j} that
\begin{equation}
  (\Exp S_x^2)/x^2 \to 0
  \label{eq:Sx2o1}
\end{equation}
as $x \to \infty.$

\subsection*{Proof of main theorem}

As in the proofs of Propositions~\ref{prop:tail} and \ref{prop:timeout}, we have that
\(
  \left( \alpha_{\lambda,t} - \alpha_{\lambda,T_x'} - 4\pi : t \geq T_x', \lambda > 0 \right)
\)
is stochastically dominated by 
\(
  \left( \alpha_{\lambda \tfrac 4 \beta \drift(T_x'),t} : t \geq 0, \lambda > 0 \right).
\)
Therefore we have by Proposition~\ref{prop:mgle} that there is a $\gamma > 0$ so that for all $C>0,$
\[
  \max_{x \leq \lambda \leq 2x}
  \Pr\left( 
  \alpha_{\lambda,T_x} - \alpha_{\lambda,T_x'} - 2\lambda( \tfrac 4 \beta )( \drift(T_x) - \drift(T_x')) \leq -C + 4\pi
  \right) \leq e^{-\gamma C^2/(T_x - T_x')}.
\]
In particular we conclude that 
\begin{equation}
  \label{eq:3rd}
  \max_{x \leq \lambda \leq 2x}
  \left\{
    -
    M_{\lambda,T_x}
    +
    M_{\lambda,T_x'}
  \right\}
  \leq C_\beta R_x (\log x)^{3/4}
\end{equation}
with probability going to $1.$

Finally, we observe that for $0 \leq \lambda \leq 2x,$
\[
  0
  \leq
  \alpha_{\lambda,\infty}
  -
  \alpha_{\lambda,T_x}
  =
  M_{\lambda,\infty}
  -
  M_{\lambda,T_x}
  +2\lambda(\tfrac 4\beta \drift(T_x))
  \leq
  M_{\lambda,\infty}
  -
  M_{\lambda,T_x}
  +\tfrac{16}{\beta}.
\]
Therefore, we conclude that 
\begin{equation}
  \max_{x \leq \lambda \leq 2x}
  \left\{ 
    M_{\lambda,\infty}
  \right\}
  \geq
  \max_{x \leq \lambda \leq 2x}
  \left\{ 
    M_{\lambda,T_x}
  \right\}
  -\frac{16}{\beta}
  \label{eq:4th}
\end{equation}
Combining \eqref{eq:2nd}, \eqref{eq:3rd} and \eqref{eq:4th}, we conclude that
\[
  \max_{x \leq \lambda \leq 2x}
  \left\{ 
    M_{\lambda,\infty}
  \right\}
  \geq \frac{4}{\sqrt{\beta}}\log(x) - C_\beta R_x (\log x)^{3/4} - (R_x^2+R_x)\sqrt{\log x} - \frac{16}{\beta}
\]
with probability going to $1$ as $x \to \infty.$ 

\bibliographystyle{abbrv}
\bibliography{sinebetamax}

\end{document}